\documentclass[conference]{IEEEtran}
\IEEEoverridecommandlockouts
\usepackage{cite}
\usepackage{amsmath,amssymb,amsfonts}
\usepackage{amsthm}

\usepackage{algorithm}
\usepackage{algorithmic}
\usepackage[table,xcdraw]{xcolor}
\newtheorem*{remark}{Remark}

\usepackage{graphicx}
\usepackage{textcomp}
\usepackage{hyperref}
\usepackage{bm}
\usepackage{tikz}
\usepackage{overpic}
\usepackage{multirow}
\usepackage{amsfonts}
\DeclareMathOperator*{\argmin}{arg\,min}
\newcommand{\R}{\mathbb{R}}

\newcommand{\ve}[1]{\mathbf{#1}}
\def\x{\ve{x}}
\def\y{\ve{y}}

\def\A{\ve{A}}
\def\D{\ve{D}}

\newtheorem{definition}{Definition}[section]

\newtheorem{Proposition}[definition]{Proposition}

\usepackage{etoolbox}
\makeatletter
\patchcmd{\@makecaption}
  {\scshape}
  {}
  {}
  {}
\makeatother
\usepackage[font=small,labelfont=bf,tableposition=top]{caption}

\newcommand{\circledstar}{%
  \tikz[baseline=(char.base)]{
    \node[shape=circle, draw, inner sep=0.1pt ] (char) {$\star$};
  }%
}

\def\BibTeX{{\rm B\kern-.05em{\sc i\kern-.025em b}\kern-.08em
    T\kern-.1667em\lower.7ex\hbox{E}\kern-.125emX}}
\begin{document}

\title{Whiteness-based bilevel learning  of regularization parameters in imaging
\thanks{The authors acknowledge the support by  the ANR JCJC TASKABILE grant ANR-22-CE48-0010.}
}
\author{Carlo Santambrogio$^{1,2}$, Monica Pragliola$^{3}$, Alessandro Lanza$^{4}$, Marco Donatelli$^{1}$, Luca Calatroni$^{2}$
	\\
$^{1}$ \emph{Science and High Technology Department, University of Insubria, Italy}\\
$^{2}$ \emph{Laboratoire I3S, CNRS, UniCA, Inria,  Sophia-Antipolis, France} \\
$^{3}$ \emph{Dept. of Mathematics and Applications, University of Naples Federico II, Italy} \\ 
$^{4}$ \emph{Dept. of Mathematics, University of Bologna, Italy}
}


\maketitle

\begin{abstract}
We consider an unsupervised bilevel optimization strategy for learning regularization parameters in the context of imaging inverse problems in the presence of additive white Gaussian noise. Compared to supervised and semi-supervised metrics relying either on the prior knowledge of reference data and/or on some (partial) knowledge on the noise statistics, the proposed approach optimizes the whiteness of the residual between the observed data and the observation model with no need of ground-truth data. We validate the approach on standard Total Variation-regularized image deconvolution problems which show that the proposed quality metric provides estimates close to the mean-square error oracle and to discrepancy-based principles.
\end{abstract}


\section{Introduction}

Variational methods are a reference paradigm in the field of ill-posed imaging inverse problems. They aim at stabilizing the unstable inversion process by minimizing a suitable energy functional encoding prior information available both on the desired image (such as sparsity, smoothness) and the noise statistics. Given a blurred, noisy and possibly incomplete (vectorized) image $\y\in\R^m$ and a linear observation model $\A\in\R^{m\times n}$, the task of retrieving a degradation-free image $\x^*\in\R^n, m\leq n$ from $\y$ in the presence of Additive White Gaussian Noise (AWGN) can be reformulated in variational terms as the optimization problem:
\begin{equation}  \label{eq:var_prob}
    \argmin_{\x\in\R^n}~ \left( F(\x;\A,\y,\bm{\lambda}):=\frac{1}{2}\|\A\x-\y\|^2 + R(\x;\bm{\lambda}) \right),
\end{equation}
where $R:\R^n\to \R_{\geq 0} \cup \left\{+\infty\right\}$ enforces prior knowledge through regularization and $\bm{\lambda}\in\Lambda \subseteq \R^{\ell}$ is a vector of hyper-parameters tailoring the amount of regularization in terms of specific local/global features and against the quadratic data term. A popular choice for $R$ consists in promoting sparsity w.r.t.~to some specific representation of the image by choosing for $\lambda>0$ $R(\x;\bm{\lambda}) = \lambda\|\mathbf{\Phi}\x\|_1$, where $\mathbf{\Phi}\in\R^{d\times n}$ or $R(\x;\bm{\lambda}) = \lambda\|\D\x\|_{2,1}$ where $\D\in\mathbb{R}^{2n\times n}$ denotes the discrete image gradient. The former choice has been the object of study in several works in the field of compressed sensing \cite{Candes2006}, while the latter has been extensively employed starting from the pioneering work \cite{ROF} under the name of Total Variation (TV) regularization  which is still a benchmark for model-based approaches for imaging. 

Note that while the choice of a scalar $\lambda>0$ essentially serves to balance the (global) action of the regularizer against the data term, local choices in the form $R(\x;\bm{\lambda})=\sum_{j=1}^n \lambda_j r_j(\x)$ enforces regularization at a local level weighted by a vector of space-variant parameters $\bm{\lambda}$, see, e.g., \cite{Calatroni_2020} for a straightforward space-variant generalization of the TV regularizer and \cite{SIAMrev2021} for some more sophisticated variants. Even in the scalar case, however, choosing an optimal $\lambda = \widehat{\lambda}\,$ is a challenging problem. Classical approaches rely on the use of cross-validation or on the (heuristic) study of the Pareto frontier as in the case of L-curve \cite{Hansen1992}. Provided that some information on the AWGN component is available (i.e., its standard deviation value $\sigma>0$), Morozov-type approaches aim at estimating $\widehat{\lambda}$ by imposing that the solution $\x^*(\widehat{\lambda})$ of \eqref{eq:var_prob} satisfies $\|\A\x^*(\widehat{\lambda}) - \y\|_2^2 \approx m\sigma^2 $ \cite{Morozov1966} or via more refined unbiased estimators \cite{Pesquet2009} depending on $\sigma$. We remark that in practice brute-force methods are still often employed whenever $\ell$ is reasonably small. However, over-parametrized variational and deep-learning models often require estimations of millions of parameters which make their use prohibitive.

Bilevel learning \cite{Haber_2003,SamuelTappen2009,pedregosa16,Kunisch2013,TuomoJMIV2017,SIG-111} is a powerful paradigm for the estimation of optimal hyper-parameters $\widehat{\bm{\lambda}}$. There, the idea consists in optimizing a certain quality measure $\mathcal{Q}(\bm{\lambda},\x^*(\bm{\lambda}))$ assessing the goodness of the solution $\x^*(\bm{\lambda})$ to \eqref{eq:var_prob} with respect to the specific task considered. In the context of imaging, standard bilevel approaches make use of classical mean-squared error (MSE) metrics mimicking SNR-type optimization \cite{Kunisch2013} and/or other ones related to more perception-inspired metrics such as the Structural Similarity Index (SSIM), see, e.g. \cite{TuomoJMIV2017}. While this choice is natural, it suffers from the major problem of requiring  reference ground-truth data for its assessment, which may be restrictive in applications. Other quality measures can be used within a semi-supervised framework providing optimal estimations depending only on information coming from the noise, in a discrepancy-based fashion \cite{Fehrenbach2015}. 

\smallskip

In this work, we propose an unsupervised bilevel learning approach where optimality of $\widehat{\bm{\lambda}}$ is assessed by maximizing the whiteness of the residual between the observations $\y$ and the observation model $\A\x^*(\widehat{\bm{\lambda}})$   \cite{etna,Pragliola2023}. As a proof of concept, we validate the approach for the case of simple TV-regularized image deconvolution problem in the scalar case $\bm{\lambda}=\lambda>0$. Our preliminary results show that the proposed whiteness measure allows to achieve results almost as good as in the fully supervised and semi-supervised case, but depending only on the model $\A$ and the measurements $\y$, making it interesting for applications where the use of a large number of examples is prohibitive.

\section{Methods}

We review in this section the main ingredients of our approach, that is the bilevel learning paradigm for the estimation of $\widehat{\bm{\lambda}}$ in both the supervised and semi-supervised case. We then propose an unsupervised metric based on residual whiteness.

\subsection{Hyper-parameter bilevel learning}

Bilevel optimization approaches compute optimal parameters $\widehat{\bm{\lambda}}$ by solving a nested optimization problem in the form
\begin{align}
    \widehat{\bm{\lambda}}\in & \argmin_{\bm{\lambda}\,\in\,\Lambda}~\sum_{k=1}^K\mathcal{Q}(\x^*_k(\bm{\lambda})) \label{eq:bilevel_sup} \\ 
& \text{s.t.}\quad\x^*_k(\bm{\lambda})\in\argmin_{\x\in\R^n}~F(\x;\A,\y_k,\bm{\lambda}), \quad k=1,\ldots,K, \notag
\end{align}
where, for $k=1,\ldots,K$, $\mathcal{Q}$ is a metric assessing the quality of the reconstructed images $\x^*_k(\bm{\lambda})$ given the measurements $\y_k$ (typically, some blurred, noisy and possibly undersampled image patches) w.r.t.~some reference quantities. 
Solving \eqref{eq:bilevel_sup} usually requires some smoothness assumption on the lower-level variational constraint described by $F$ so as to allow implicit differentiation w.r.t.~$\bm{\lambda}$ by means of the implicit function theorem, see, e.g.~\cite{Kunisch2013,Haber_2003} and Section \ref{sec:opt_bil}. 

\medskip

Regarding the choice of $\mathcal{Q}$, supervised approaches (see, e.g., \cite{Kunisch2013,TuomoJMIV2017}) make use of pairs of exemplar ground-truth images corresponding to the measurements $\left\{\bar{\x}_k,\y_k\right\}$, $k=1,\ldots,K$ so that $\mathcal{Q}$
enforces proximity between $\x^*_k(\bm{\lambda})$ and $\bar{\x}_k$, i.e.:
\begin{equation}  \label{eq:supervised_loss}
\mathcal{Q}_{\text{MSE}}(\x^*(\bm{\lambda});\bar{\x}): = \frac{1}{2}\| \x^*(\bm{\lambda})-\bar{\x}\|_2^2,
\end{equation}
which practically corresponds to choose $\bm{\lambda}$ so as to optimize the Signal to Noise Ratio (SNR) of the reconstructions. 

In \cite{Fehrenbach2015} a semi-supervised quality metric was employed to select the optimal $\widehat{\bm{\lambda}}$. Differently from \eqref{eq:supervised_loss}, the quality loss employed therein relates more to a discrepancy-type measure assessing whether the residual quantity $\bm{r}(\bm{\lambda}):=\A\x^*(\bm{\lambda})-\y$ has magnitude close to the noise intensity, i.e.~$\|\bm{r}(\bm{\lambda}) \|_2^2 \approx m\sigma^2 $, so that a natural semi-supervised choice for dealing with AWGN which does not require the ground truth data $\left\{\bar{\x}_k\right\}$ but an estimate of $\sigma$ is the Gaussianity loss:
\begin{equation}  \label{eq:gaussianity}
\mathcal{Q}_{\text{Gauss}}(\x^*(\bm{\lambda});\sigma) := \frac{1}{2}\left( \|\bm{r}(\bm{\lambda}) \|_2^2 - m\sigma^2  \right)^2.
\end{equation}

\subsection{Residual whiteness principle}

To avoid the prior knowledge of both the reference data ${\bar{\x}_k}$ and the Gaussian noise standard deviation $\sigma$, a different quality measure optimizing the whiteness of the residual $\bm{r}(\bm{\lambda})$ can be used. Employed in an heuristic fashion under the name of \emph{residual whiteness principle} in several papers, see, e.g., \cite{etna,Pragliola2023}, we propose here to use such loss as an unsupervised quality metric for \eqref{eq:bilevel_sup}. We thus consider:
\begin{equation} \label{eq:whiteness}
\mathcal{Q}_{\text{White}}(\x^*(\bm{\lambda})):= \frac{1}{2}\Big\| \frac{\bm{r}(\bm{\lambda})~ \circledstar~  \bm{r}(\bm{\lambda})}{\|\bm{r}(\bm{\lambda})\|_2^2} \Big\|_2^2,    
\end{equation}
where, with a little abuse of notation, we denote by $\x_1~\circledstar~\x_2$ the discrete circular cross-correlation between the matrices $\bm{X}_1, \bm{X}_2\in\R^{n_1\times n_2}$ such that $\text{vec}(\bm{X}_d) = \x_d\in\R^n$, $d=1,2$ and $n_1n_2=n$ which is defined for $(j_1,j_2)\in \left\{0,\ldots,n_1-1\right\}\times\left\{0,\ldots,n_2-1\right\}$  by
\small
\[
(\bm{X}_1~\circledstar~\bm{X}_2)_{j_1,j_2} = \sum_{k_1=0}^{n_1-1} \sum_{k_2=0}^{n_2-1} (\bm{X}_1)_{k_1,k_2}(\bm{X}_2)_{(j_1+k_1)\text{mod }n_1,(j_2+k_2)\text{mod }n_2}.
\]
\normalsize
Note that as observed in \cite{etna} the normalization term in \eqref{eq:whiteness} eliminates the dependence on $\sigma$.

We remark that to deal with noise scenarios different than AWGN, in \cite{Bevilacqua2023} a whiteness measure tailored for Poisson noise was considered. {We thus expect that, under suitable modifications, our proposed approach could suit to more general noise scenarios as well for tailored choice of $\mathcal{Q}_{\text{White}}$.}

\section{Optimization algorithms}  \label{sec:opt_bil}

In this section we discuss the algorithms employed to solve both the lower- and the upper-level optimization problems in \eqref{eq:bilevel_sup}. For simplicity we consider in the following discussion $\Lambda=\R_{>0}$ and $K=1$, that is we look for an optimal positive parameter $\widehat{\lambda}$ based only on the observed image $\y$.  In order to exploit implicit differentiation for the computation of the gradient of the bilevel problem \eqref{eq:bilevel_sup}, we consider a smoothed version $F_\varepsilon$ of the $\ell_2$-TV functional, defined in terms of a $C^2$ Huber smoothing of the TV regularization term given by:
\begin{equation}
H_\varepsilon(\D\x):=\| \D\x \|_{2,1,\varepsilon} = \sum_{j=1}^n h_\varepsilon(  (\D\x)_j ),
\label{eq:TV_smoothed}
\end{equation}
where $(\D\x)_j\in \R^2$ is the $j$-th component of the discrete image gradient $\D\x\in \R^{2n}$ and $h_\varepsilon:\R^2\to \R_{\geq 0}$ is a $C^2$ Huber smoothing function defined by:
\begin{equation}   \label{eq:huber_def}
h_{\varepsilon}(\bm{v}) = \begin{cases}
    \frac{3}{4\varepsilon}\|\bm{v}
\|_2^2 - \frac{1}{8\varepsilon^3}\|\bm{v}\|_2^4\quad&\quad\text{if }\|\bm{v}\|_2<\varepsilon\\
\|\bm{v}\|_2 - \frac{3\varepsilon}{8}\quad&\quad\text{if }\|\bm{v}\|_2\geq\varepsilon\end{cases}.
\end{equation}
The smoothed $\ell_2$-TV functional
\begin{equation} \label{eq:smooth_Feps}
F_\varepsilon(\x;\A,\y,\lambda) := \frac{1}{2}\|\A\x-\y\|_2^2 + \lambda H_\varepsilon(\D\x)
\end{equation}
is thus $C^2$. Its optimization properties are reported in the following proposition.

\begin{Proposition}
The functional $F_{\epsilon}$ defined in (\ref{eq:TV_smoothed})-(\ref{eq:smooth_Feps}) 
is twice continuously differentiable and convex on $\R^n$. Moreover, if ${\rm ker}(\A)\cap {\ker}(\D) = \{\bm{0}_n\}$, $F_{\varepsilon}$ is also coercive and, hence, admits a compact convex set of global minimisers. Its gradient $\nabla_{\x} F_{\varepsilon} \in \R^n$ and  Hessian $\nabla^2_{\x}  F_{\varepsilon} \in \R^{n \times n}$ are given by
\begin{eqnarray}
\nabla_{\x}\, F_{\varepsilon}(\x;\A,\y,\lambda)
&\!\!\!{=}\!\!\!&
\bm{\mathrm{A}}^{\mathrm{T}}\left(  \bm{\mathrm{A}} \x - \y \right)
\:{+}\;
\lambda\: \bm{\mathrm{D}}^{\mathrm{T}} \, 
\nabla H_{\varepsilon}(\bm{\mathrm{D}}\x),
\nonumber\\
\nabla^2_{\x}\, F_{\varepsilon}(\x;\A,\y,\lambda)
&\!\!\!{=}\!\!\!&
\bm{\mathrm{A}}^{\mathrm{T}} \bm{\mathrm{A}}
\:{+}\;
\lambda \: 
\D^{\mathrm{T}}\, 
\nabla^2  H_{\varepsilon}
(\bm{\mathrm{D}}\x) \, \D,
\label{eq:J_Hess}
\end{eqnarray}
where $\nabla  H_{\varepsilon}$ and $\nabla^2  H_{\varepsilon}$ denote the vector (respectively, matrix) of first-(respectively, second-)order derivatives of $H_\varepsilon(\bm{z})$ w.r.t.~$\bm{z}=\D\x$.
For any $\lambda,\varepsilon \in \R_{++}$, the gradient $\nabla_{\x} F_{\varepsilon}$ is thus $L_{\lambda,\varepsilon}$-Lipschitz continuous, 
%
and $L_{\lambda,\varepsilon}$ can be bounded as:
\begin{equation}
\small
L_{\,\lambda,\varepsilon} 
\:{=}\:
\max_{\x \in \R^n} \left\| \nabla^2_{\x} F_{\varepsilon}(\x;\A,\y,\lambda) \right\|_2
\:{\leq}\;
\left\| \bm{\mathrm{A}} \right\|_2^2 + 
\frac{12 \, \lambda}{\varepsilon}
\;{=:}\;\,
\overline{L}_{\lambda,\varepsilon} .
\label{eq:L_bound}
\end{equation}
\end{Proposition}

\begin{proof}
It follows easily from definitions (\ref{eq:TV_smoothed})-(\ref{eq:smooth_Feps}) 
that $F_{\varepsilon}$ is $C^2(\R^n)$ and convex on $\R^n$. Then, if $\text{ker}(\A)\cap \text{ker}(\D) = \{\bm{0}_n\}$, $F_{\varepsilon}$ is coercive as both the fidelity and regularization terms are compositions of a linear map - with coefficient matrix $\A$ and $\D$, respectively - and a coercive function. It follows that $F_{\varepsilon}$ admits a compact convex set of global minimizers. The gradient and Hessian expressions in \eqref{eq:J_Hess} are derived easily by applying the chain rule of differentiation (Jacobian of composite functions). Then, based on \eqref{eq:J_Hess} and on the sub-additivity and sub-multiplicativity properties of the spectral matrix norm, the smallest gradient Lipschitz constant $L_{\lambda,\varepsilon}$ in \eqref{eq:L_bound} satisfies
\begin{eqnarray}
L_{\lambda,\varepsilon} 
&\!\!\!\!{=}\!\!\!&
\max_{\x \in \R^n} \, \left\|  \bm{\mathrm{A}}^{\mathrm{T}} \bm{\mathrm{A}}
\:{+}\;
\lambda \, \bm{\mathrm{D}}^{\mathrm{T}} 
\nabla^2 
H_{\varepsilon}(\bm{\mathrm{D}}\x) \bm{\mathrm{D}} \, \right\|_2
\nonumber\\
&\!\!\!\!{\leq}\!\!\!&
\left\| \bm{\mathrm{A}} \right\|_2^2 + 
\lambda \, \left\| \bm{\mathrm{D}} \right\|_2^2 \,
\max_{\x \in \R^n} \, 
\left\|
\nabla^2
H_{\varepsilon}(\bm{\mathrm{D}}\x)\right\|_2
\nonumber\\
&\!\!\!\!{\leq}\!\!\!&
\left\| \bm{\mathrm{A}} \right\|_2^2 + 
8 \, \lambda \,  
\max_{\x \in \R^n} \, 
\left\|
\nabla^2 
H_{\varepsilon}(\bm{\mathrm{D}}\x)\right\|_2  \, ,
\label{eq:LL}
\end{eqnarray}
where \eqref{eq:LL} comes from recalling that $\left\| \bm{\mathrm{D}} \right\|_2^2 \leq 8$ (with $\left\| \bm{\mathrm{D}} \right\|_2^2 \approx 8$) when the discrete gradient operator $\D$ approximates horizontal and vertical partial derivatives by means of (unscaled) standard finite differences \cite{Chambolle2004}. 
It can be proved (we omit the proof due to the page limit) that $\nabla^2 
H_{\varepsilon}(\bm{\mathrm{D}}\x) \in \R^{2n \times 2n}$ is a real symmetric $2 \times 2$ block matrix with diagonal blocks which admits $\x$-dependent eigenvalue decomposition
\begin{equation}
\nabla^2 
H_{\varepsilon}(\bm{\mathrm{D}}\x)
\,\;{=}\; 
\bm{\mathrm{V}}_{\varepsilon}^{\mathrm{T}}(\D\x) \,
\bm{\mathrm{E}}_{\varepsilon} (\D\x) \,
\bm{\mathrm{V}}_{\varepsilon}(\D\x) \, ,
\end{equation}
with orthogonal modal matrix $\bm{\mathrm{V}}_{\varepsilon}(\D\x)$ and eigenvalue matrix $\bm{\mathrm{E}}_{\varepsilon}(\D\x) = \mathrm{diag}(e_{\varepsilon}^{(1)}(\D\x),\ldots,e_{\varepsilon}^{(2n)}(\D\x))$ satisfying
\begin{equation}
0 \leq e_{\varepsilon}^{(i)}(\D\x) \leq \frac{3}{2 \, \varepsilon}\,,  \;\;\; \forall \, i = 1,\ldots,2n \, , \;\; \forall \, \x \in \R^n \, .
\end{equation}
It thus follows that
\begin{equation}
\max_{\x \in \R^n} \, 
\left\|
\nabla^2 
H_{\varepsilon}(\bm{\mathrm{D}}\x)\right\|_2
\;{=}\;
\max_{\x \in \R^n} \,
\left\| \bm{\mathrm{E}}_{\varepsilon}(\D\x)\right\|_2
\;{=}\;
\frac{3}{2\,\varepsilon} \, ,
\label{eq:E_bound}
\end{equation}
where the last equality comes from the fact that there always exists $\x \in \R^n$ such that at least one among the eigenvalues is equal to the upper bound $3 / (2 \, \varepsilon)$.
By replacing \eqref{eq:E_bound} into \eqref{eq:LL}, we obtain 
\eqref{eq:L_bound}.
\end{proof}
\normalsize

\begin{remark}
We remark that for many inverse imaging problems the value $\left\| \bm{\mathrm{A}} \right\|_2^2$ in \eqref{eq:L_bound} is known a priori or easily computable. For instance, $\left\| \bm{\mathrm{A}} \right\|_2^2 = 1$ in image deblurring (with normalized blur functions), inpainting 
problems.
\end{remark}

\medskip

We now want to make precise the optimization algorithms required to solve both the lower and the upper-level problem.  To do that, and similarly as in \cite{Fehrenbach2015,pedregosa16},  we consider at first a change of variables $\lambda = \exp{(\beta)}$ in order to deal with an unconstrained problem depending on a parameter $\beta\in\R$.  Upon this choice note that there holds:
\[
\mathcal{Q}(\x^*(\lambda)) = \mathcal{Q}(\x^*(w(\beta))),
\]
where $w:\beta\mapsto \lambda$ is the exponential function.  Neglecting for simplicity the dependence of $F_\varepsilon$ on the problem ingredients $\A$ and $\y$ and denoting for ease of notation by $\x^*=\x^*(\lambda)=\x^*(w(\beta))$ the solution of the lower-level problem in \eqref{eq:bilevel_sup} for a fixed $\lambda$, by optimality we get:
\begin{equation}  \label{eq:optimality}
\nabla_{\x} ~F_\varepsilon(\x^*;w(\beta)) = \bm{0},
\end{equation}
which, by differentiating the left-hand-side w.r.t.~$\beta$ and applying the chain rule entails:
\begin{align}
& \frac{\partial \nabla_{\x}~ F_\varepsilon(\x^*;w(\beta))}{\partial \beta}  = \frac{\partial \nabla_{\x}  F_\varepsilon(\x^*;w(\beta))}{\partial w}\frac{\partial w}{\partial \beta} \notag \\
& + 
 \nabla^2_{\x} F_\varepsilon(\x^*;w(\beta)) \frac{\partial \x^*}{\partial \beta},
\end{align}
whence, by \eqref{eq:optimality} and the implicit function theorem, entails:
\[
 \frac{\partial\x^*}{\partial \beta} = -\left( \nabla^2_{\x} F_\varepsilon(\x^*;w(\beta)) \right)^{-1} \frac{\partial \nabla_{\x}  F_\varepsilon(\x^*;w(\beta)}{\partial w}\frac{\partial w}{\partial \beta},
\]
which can be used for the computation of the derivative  of the nested problem so as to get:
\begin{align}
&\frac{\partial \mathcal{Q}(\x^*)}{\partial \beta} =  \nabla_{\x}\mathcal{Q}(\x^*)^T\frac{\partial \x^*}{\partial \beta}   \label{eq:composite_derivative} \\
& = -\nabla_{\x}\mathcal{Q}(\x^*)^T\left( \nabla^2_{\x} F_\varepsilon(\x^*;w(\beta)) \right)^{-1} \frac{\partial \nabla_{\x}  F_\varepsilon(\x^*;w(\beta))}{\partial w}\frac{\partial w}{\partial \beta},  \notag 
\end{align}
where 
the expression of $\nabla_{\x}\mathcal{Q}(\x^*(\lambda))$ depends on the specific expression of the assessment loss considered. Formula \eqref{eq:composite_derivative} is classically used for standard gradient-type algorithms (such as gradient-descent, quasi-Newton\ldots) addressing the bilevel problem \eqref{eq:bilevel_sup}. 

We now describe in Section \ref{sec:AGD} the accelerated first-order solver of the lower-level problem computing (an approximation of) $\x^*$ and describe in Section \ref{sec:GN} a Gauss-Newton strategy using an expression similar to \eqref{eq:composite_derivative} to address the solution of the nested bilevel problem.



\subsection{Accelerated gradient-descent lower-level solver} \label{sec:AGD}

To compute (approximate) minimizers $\x^*$ of the smooth and convex functional $F_\varepsilon$ in \eqref{eq:smooth_Feps} we consider the Nesterov's accelerated gradient-descent algorithm, see Algorithm \ref{algo:Nesterov}. We preferred such approach to Newton-type techniques in order to reduce the computational costs required for the inversion of (approximations of) the Hessian along the iterations. Still, to get a good approximation of $\x^*$ a fairly small relative tolerance parameter $\epsilon$ should be employed to assess optimality. Note that at each iteration $i\geq 1$ of the outer optimization solver, a fixed step-size $L_i=\overline{L}_{\lambda_i,\varepsilon}$ in \eqref{eq:L_bound} depending on the current estimate $\lambda_i$ is used.

\begin{algorithm}
\caption{Nesterov AGD,~ \texttt{Nesterov$_\text{AGD}$}$(\x_0,L_i,\lambda_i,\epsilon)$}\label{algo:Nesterov}
\begin{algorithmic}
\STATE \textbf{Initialize:} $\theta_0=1, \tau=1/L_i,~ \x_{-1}=\x_0, t=0$
\WHILE{$\| \x_{t+1} - \x_{t}\|_2 > \epsilon$}
    \STATE $\theta_{t+1}= \frac{1+\sqrt{1+4\theta_t^2}}{2}$
    \STATE $\bm{z}_{t+1} = \x_t + \frac{\theta_t-1}{\theta_{t+1}}(\x_t-\x_{t-1})$
    \STATE $\x_{t+1} = \bm{z}_{t+1} - \tau \nabla_{\x} F_{\varepsilon}(\bm{z}_{t+1};\A,\y,\lambda_i)$
    \STATE $t=t+1$
\ENDWHILE
\RETURN $\x^*=\x^*(\lambda_i)$
\end{algorithmic}
\end{algorithm}

\subsection{Gauss-Newton upper-level solver}  \label{sec:GN}

We now describe the optimization algorithm solving the outer problem in \eqref{eq:bilevel_sup} for optimizing over $\lambda$ the different quality losses $\mathcal{Q}$ described above. Note that independently on the convexity of the loss function $\mathcal{Q}$ (which holds for instance both for \eqref{eq:supervised_loss} and \eqref{eq:gaussianity}), the nested problem is generally non-convex hence only convergence to local minima $\widehat{\lambda}=\exp{(\widehat{\beta})}$ is expected. In the literature, several algorithms have addressed this task. In \cite{Kunisch2013,TuomoJMIV2017}, for instance, a semi-smooth Newton algorithm was employed. Here, similarly as in \cite{Fehrenbach2015} we employ a Gauss-Newton algorithm which suits well to the squared $\ell^2$-type structure of the three losses employed which can indeed be all expressed as 
\begin{equation} \label{eq:Q_rho}
    \mathcal{Q}(\x^*(\lambda))=\frac{1}{2}\| \rho(\x^*(\lambda)) \|_2^2
\end{equation} 
for suitable choices and dimensionality of $\rho(\x^*(\lambda))$ depending on the choice of $\mathcal{Q}\in\left\{\mathcal{Q}_{\text{MSE}},\mathcal{Q}_{\text{Gauss}}, \mathcal{Q}_{\text{White}} \right\}$. By \eqref{eq:Q_rho}, we observe that
\[
\frac{\partial \mathcal{Q}(\x^*(\exp{(\beta_i)}))}{ \partial \beta} = \rho(\x^*(\exp{(\beta_i)}))^T \bm{J}_{\rho}(\beta_i),
\]
where $\bm{J}_{\rho}(\beta_i) = \frac{\partial \rho (\x^*(\exp{(\beta_i)}))}{\partial \beta}$ can be computed similarly as in \eqref{eq:composite_derivative} depending on the particular choice of $\mathcal{\rho}$. We can now define in Algorithm \ref{algo:GN} a Gauss-Newton solver for the bilevel problem \eqref{eq:bilevel_sup} making explicit use of the residual function $\rho(\cdot)$. The algorithm depends on a tolerance parameter $\epsilon_1$ which assesses stationarity and a line-search parameter $\alpha\in(0,1)$ which could potentially be estimated on the fly by imposing, e.g., Wolfe-type conditions, but which we preferred to fix beforehand.

\begin{algorithm}
\caption{\small Gauss-Newton bilevel solver, \texttt{GN}$_\text{bil}(\beta_0,\epsilon_1,\alpha,\text{max\_it})$}   \label{algo:GN}
\begin{algorithmic}
\STATE \textbf{Initialize:} $i=0$, $\x^*(w(\beta_{-1}))=\y$
\WHILE{$\| d_i\|_2 > \epsilon_1$ and $i < \text{max\_it}$}

    \STATE \small{compute~$\x^*(w(\beta_i))$=\texttt{Nesterov}$_\text{AGD}(\x^*(w(\beta_{i-1})),L_i,w(\beta_i),\epsilon)$}
    
    \STATE \small{using \eqref{eq:composite_derivative}, compute descent direction by solving}:
    \[
     d_i = - \left( \bm{J}_{\rho}
     (\beta_i)^T \bm{J}_{\rho}(\beta_i) \right)^{-1} \left( \bm{J}_{\rho}(\beta_i)^T \rho(\x^*(w(\beta_i)))\right)
    \]
    
    \STATE \small{update using} $
     \beta_{i+1}= \beta_i + \alpha d_i 
     $

    \STATE $i=i+1$
\ENDWHILE
\RETURN $\widehat{\lambda} = w{(\widehat{\beta})}$
\end{algorithmic}
\end{algorithm}

\normalsize

\section{Experimental results}

We now compare the proposed bilevel approaches on two exemplar image deconvolution problems with different types of blur (Gaussian, motion) and AWGN of different magnitude. The algorithmic parameters for both Algorithm \ref{algo:Nesterov} and \ref{algo:GN} are chosen as $\epsilon=10^{-6}$, $\epsilon_1=10^{-5}$, $\alpha=0.1$, max\_it = $60$. The Huber smoothing parameter in \eqref{eq:huber_def} is chosen as $\varepsilon=10^{-3}$, while the initial $\beta$-value in Algoritgm \ref{algo:GN} has been set as $\beta_0=2$.


For our tests, we considered a dataset of 30 test images with size 180$\times$180 pixels from the BSD400 repository \cite{BSD400}. The generic test image $\bar{\x}$ has been corrupted by space-invariant blur defined by a convolution kernel;
 the Gaussian blur kernel used has square support of side 9 pixels and standard deviation 2,  while for the motion blur kernel the support size is 10 pixels and the direction angle is 60$^{\circ}$. The generic blurred image $\A\bar{\x}$ has then been corrupted by realizations of AWGN with different magnitudes. In our set-up, the noise level is quantified by the Blurred Signal-to-Noise Ratio (BSNR) defined by:
\begin{equation}\label{eq:BSNR}
    \mathrm{BSNR}(\y,\bar{\x})=10\log_{10}\frac{\|  \A\bar{\x}-\mathbb{E}({\A\bar{\x}})  \|_2^2}{\|\A\bar{\x} - \y  \|_2^2}\,,
\end{equation}
where $\mathbb{E}({\A\bar{\x}})$ denotes the average intensity of the blurred image $\A\bar{\x}$. Notice that there exists a one-to-one \emph{inverse} relationship between the BSNR and the noise standard deviation and the larger the noise the smaller the BSNR. We thus selected increasing  values of $\mathrm{BSNR}\in\{10,17.5,25,32.5,40\}$.

For each test image and BSNR value we evaluated the quality of the image reconstructed by bilevel optimization \eqref{eq:bilevel_sup} of the regularization parameter $\lambda$ when using the three quality metrics \eqref{eq:supervised_loss} (supervised, S), \eqref{eq:gaussianity} (semi-supervised, SS) and \eqref{eq:whiteness} (unsupervised, U) in terms of average values of PSNR and SSIM computed over the set of images.

In Figure \ref{fig:out_gaussian} we report the 
results obtained in the case of Gaussian blur for the estimation of a scalar parameter $\widehat{\lambda}$. We observe that compared to the MSE oracle and the semi-supervised Gaussianity loss, the proposed whiteness-based procedure provides results which are as good but do not require the use of prior information. 

\begin{figure}[H] 
    \centering
    \begin{tabular}{c}
       \includegraphics[width=0.23\textwidth]{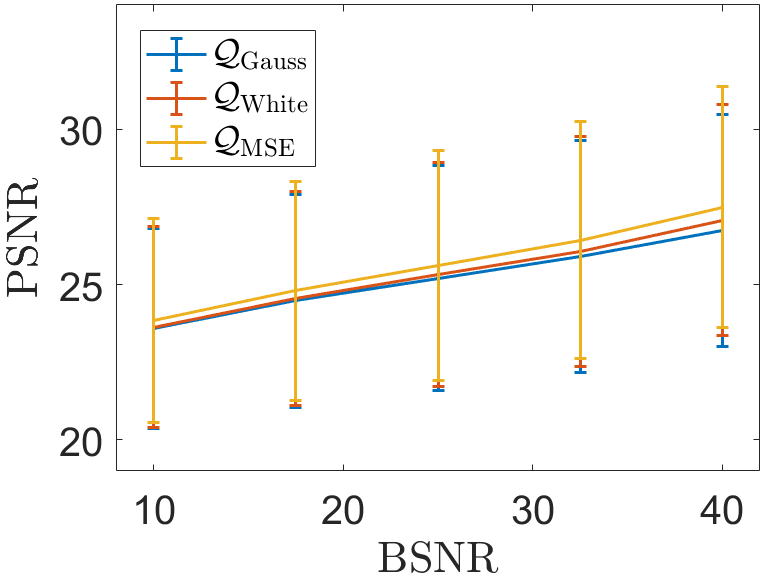} 
         \includegraphics[width=0.23\textwidth]{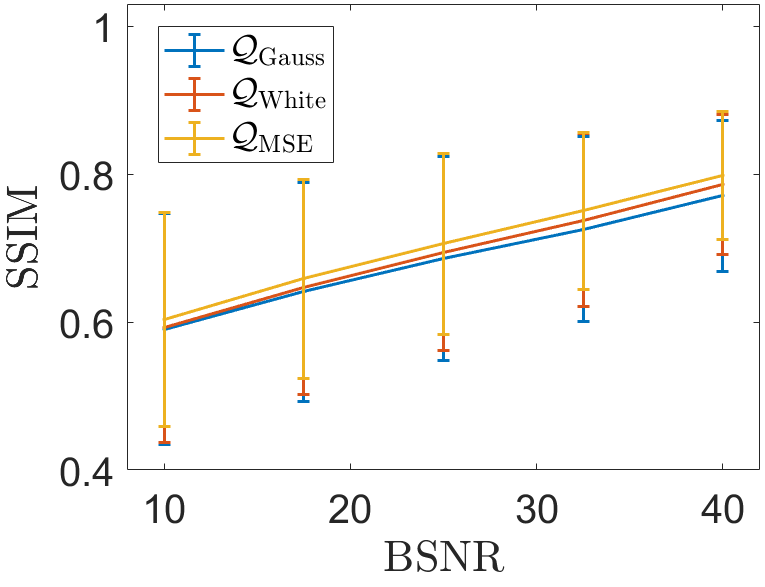}
    \end{tabular}
    \caption{Average PSNR and SSIM and dispersion bands for MSE (S), Gaussianity (SS) and Whiteness (U) loss computed over 30 test images corrupted by Gaussian blur and AWGN of different levels.}
    \label{fig:out_gaussian}
\end{figure}

A similar behavior is observed for a TV deconvolution problem in the presence of motion blur, see Figure \ref{fig:out_motion}. For this second test, we report in Table \ref{tab:1} and in Figure \ref{fig:imout} the numerical values and the visual results obtained for three different images in the dataset, respectively.

\begin{figure}[H]  
    \centering
    \begin{tabular}{c}
       \includegraphics[width=0.23\textwidth]{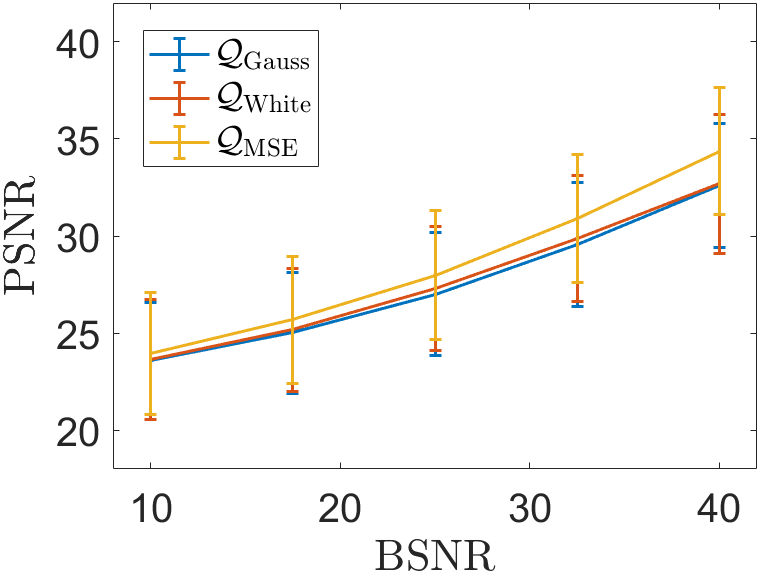} 
         \includegraphics[width=0.23\textwidth]{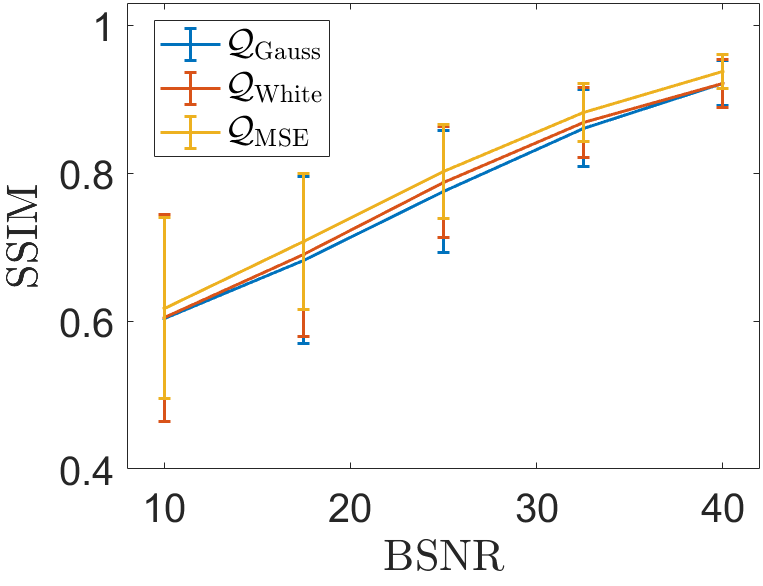}
    \end{tabular}
    \caption{
    Average PSNR and SSIM and dispersion bands  for MSE (S), Gaussianity (SS) and Whiteness (U) loss computed over 30 test images corrupted by motion blur and AWGN of different levels.}
    \label{fig:out_motion}
\end{figure}

\begin{table}  \label{table:assess}
        \centering
        \renewcommand{\tabcolsep}{0.03cm}
        \begin{tabular}{c|c|c|c|c|c|c|}
            \cline{2-7}
            \cline{2-7}
             & \multicolumn{2}{|c|}{\cellcolor[HTML]{ECF4FF} MSE (S)} & \multicolumn{2}{|c|}{\cellcolor[HTML]{ECF4FF} Gaussianity (SS)} & \multicolumn{2}{|c|}{ \cellcolor[HTML]{ECF4FF} Whiteness (U)}\\
             \cline{2-7}
            \cline{2-7}
             & \cellcolor[HTML]{FFFFC7}PSNR & \cellcolor[HTML]{FFFFC7}SSIM & \cellcolor[HTML]{FFFFC7}PSNR & \cellcolor[HTML]{FFFFC7}SSIM & \cellcolor[HTML]{FFFFC7}PSNR & \cellcolor[HTML]{FFFFC7}SSIM  \\
            \hline
            \multicolumn{1}{|c|}{\#1} &\textbf{26.50} & \textbf{0.66} &26.03 (1.8\%) &0.63 (4.1\%)&26.05 (1.7\%) & 0.64 (3.8\%)   \\
             \hline
            \multicolumn{1}{|c|}{\#13} &\textbf{22.67} & \textbf{0.54}&  22.43 (1.1\%)& 0.51 (6.1\%) & 22.50 (0.7\%)& 0.52 (4.6\%) \\
             \hline
            \multicolumn{1}{|c|}{\#22} &\textbf{19.50} & 0.65 (8.8\%) &18.91 (3.0\%) & 0.70 (0.3\%) & 19.02 (2.4\%) & \textbf{0.71} \\
            \hline
        \end{tabular}
        \vspace{0.2cm}
        \flushleft
           \caption{PSNR and SSIM achieved by optimizing the MSE (S), Gaussianity (SS) and Whiteness (U) loss for the three different images in Figure \ref{fig:imout} of the BSD400 dataset corrupted by motion blur and AWGN with BSNR=10. Percentages  w.r.t.~the maximum values (in bold) are reported.}
        \label{tab:1}
    \end{table}

\begin{figure} \label{fig:results}
    \centering
    \renewcommand{\tabcolsep}{0.1cm}
\begin{tabular}{cccc}
&\#1&\#13&\#22\\
\raisebox{.5cm}{\rotatebox{90}{ \footnotesize  GT }}&\includegraphics[width=2cm]{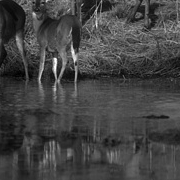}&\includegraphics[width=2cm]{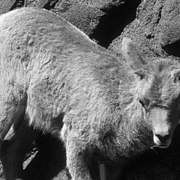}&\includegraphics[width=2cm]{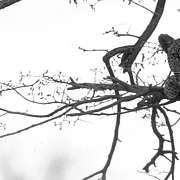}\\
\raisebox{.1cm}{\rotatebox{90}{\footnotesize  $\y$ and kernel $\bm{h}$ }}&\begin{overpic}[width=2cm]{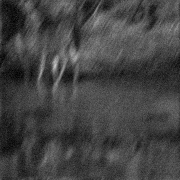}
\put(1,1){\color{red}%
\frame{\includegraphics[scale=0.05]{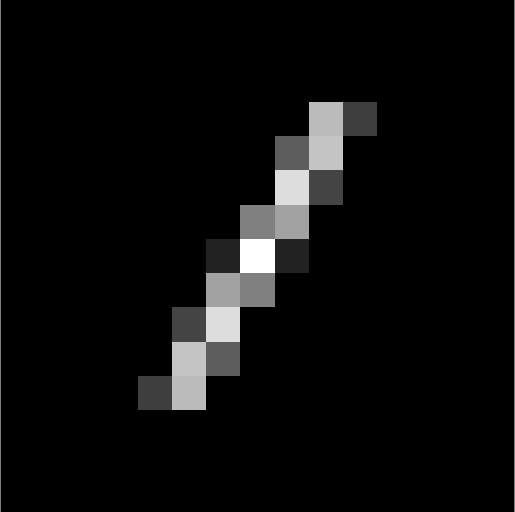}}}
\end{overpic}&\begin{overpic}[width=2cm]{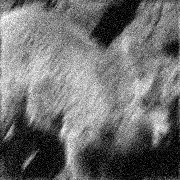}
\put(1,1){\color{red}%
\frame{\includegraphics[scale=0.05]{figs/motion_psf.png}}}
\end{overpic}&\begin{overpic}[width=2cm]{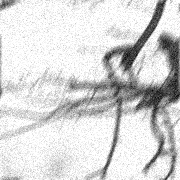}
\put(1,1){\color{red}%
\frame{\includegraphics[scale=0.05]{figs/motion_psf.png}}}
\end{overpic}\\
\raisebox{.5cm}{\rotatebox{90}{ \footnotesize MSE (S)}}&\includegraphics[width=2cm]{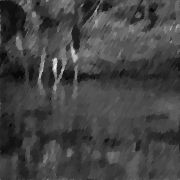}&\includegraphics[width=2cm]{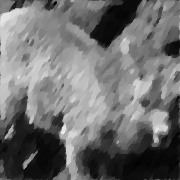}&\includegraphics[width=2cm]{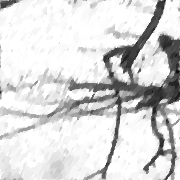}\\
\raisebox{.1cm}{\rotatebox{90}{ \footnotesize Gaussianity (SS)}}&\includegraphics[width=2cm]{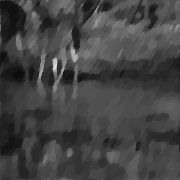}&\includegraphics[width=2cm]{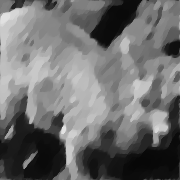}&\includegraphics[width=2cm]{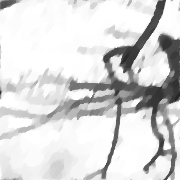}\\
\raisebox{.2cm}{\rotatebox{90}{\footnotesize Whiteness (U)}}&\includegraphics[width=2cm]{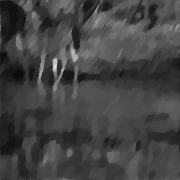}&
\includegraphics[width=2cm]{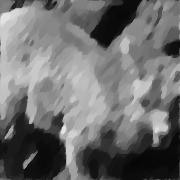}
&\includegraphics[width=2cm]{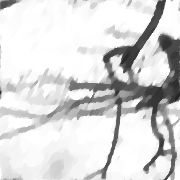}
\end{tabular}    \caption{From top to bottom, row-wise: ground-truth images, data $\y$ corrupted by motion blur and AWGN with BSNR=10, optimal reconstructions achieved by bilevel optimization of MSE (S), Gaussianity (SS) and Whiteness (U) loss. }
    \label{fig:imout}
\end{figure}

\section{Conclusions}
We proposed an unsupervised bilevel learning strategy based on residual whiteness for estimating the regularization parameters in exemplar TV-regularized image deconvolution problems in the presence of AWGN. Our results suggest that such quality measure performs as well as standard MSE-based and discrepancy-type alternatives, but does not rely on any ground truth data nor noise magnitude estimation.  Further work should address its use in more challenging inverse problems and comparisons with recent approaches \cite{Floquet2024} proposed in the context of deep learning  for, e.g., image denoising.

\bibliographystyle{IEEEtran}
\bibliography{refs}


\end{document}